\documentclass[12pt]{amsart}
\usepackage{amssymb,amscd}
\usepackage{verbatim}

\usepackage{amsmath,amssymb,graphicx,mathrsfs}   
\usepackage[colorlinks=true,allcolors = blue]{hyperref} 

\let\frak\mathfrak

\def\>{\relax\ifmmode\mskip.666667\thinmuskip\relax\else\kern.111111em\fi}
\def\<{\relax\ifmmode\mskip-.333333\thinmuskip\relax\else\kern-.0555556em\fi}
\def\vsk#1>{\vskip#1\baselineskip}
\def\vv#1>{\vadjust{\vsk#1>}\ignorespaces}
\def\vvn#1>{\vadjust{\nobreak\vsk#1>\nobreak}\ignorespaces}

  \let\ssize\scriptstyle
\let\sssize\scriptscriptstyle

\let\Medskip\medskip
\def\medskip{\par\Medskip}
\let\Bigskip\bigskip
\def\bigskip{\par\Bigskip}

\let\Maketitle\maketitle
\def\maketitle{\Maketitle\thispagestyle{empty}\let\maketitle\empty}

\newtheorem{thm}{Theorem}[section]
\newtheorem{cor}[thm]{Corollary}
\newtheorem{lem}[thm]{Lemma}

\newtheorem{defn}[thm]{Definition}

\theoremstyle{definition}                                  
\newtheorem{exmp}{Example}[section]

\numberwithin{equation}{section}

\theoremstyle{definition}
\newtheorem*{rem}{Remark}

\let\mc\mathcal
\let\nc\newcommand

\let\al\alpha

\let\ka\kappa
\let\la\lambda

\let\phi\varphi

\let\om\omega

\let\der\partial

\let\ox\otimes

\let\geq\geqslant

\let\leq\leqslant

\let\on\operatorname
\let\bi\bibitem
\let\bs\boldsymbol

\def\C{{\mathbb C}}
\def\Z{{\mathbb Z}}

\def\F{{\mc F}}

\def\+#1{^{\{#1\}}}

\def\End{\on{End}}

\def\beq{\begin{equation}}
\def\eeq{\end{equation}}
\def\be{\begin{equation*}}
\def\ee{\end{equation*}}

\nc{\bea}{\begin{eqnarray*}}
\nc{\eea}{\end{eqnarray*}}
\nc{\bean}{\begin{eqnarray}}
\nc{\eean}{\end{eqnarray}}
\nc{\Ref}[1]{{\rm(\ref{#1})}}

\let\ga\gamma
\let\Ga\Gamma

\nc{\Il}{{\mc I_{\bs\la}}}
\nc{\bla}{{\bs\la}}
\nc{\Fla}{\F_\bla}
\nc{\tfl}{{T^*\Fla}}
\nc{\GL}{{GL_n(\C)}}
\nc{\GLC}{{GL_n(\C)\times\C^*}}

\let\sd s 

\def\ddk_#1{\kk_{#1}\<\>\frac\der{\der\<\>\kk_{#1}}}

\def\bul{\mathbin{\raise.2ex\hbox{$\sssize\bullet$}}}
\def\intt{\mathchoice
{\mathop{\raise.2ex\rlap{$\,\,\ssize\backslash$}{\intop}}\nolimits}
{\mathop{\raise.3ex\rlap{$\,\sssize\backslash$}{\intop}}\nolimits}
{\mathop{\raise.1ex\rlap{$\sssize\>\backslash$}{\intop}}\nolimits}
{\mathop{\rlap{$\sssize\<\>\backslash$}{\intop}}\nolimits}}

\let\kk q 
\let\cc c

\let\Ko K

\def\GZ/{Gelfand-Zetlin}
\def\KZ/{{\slshape KZ\/}}
\def\qKZ/{{\slshape qKZ\/}}
\def\XXX/{{\slshape XXX\/}}

\nc{\A}{{\mc C}}

\def\FFF{{\mathbb F}}
\def\Sing{{\on{Sing}}}

\newcommand{\OS}{\mathcal {A}}

\def\FF{{\mc F}}

\DeclareMathOperator{\codim}{codim}

\newcommand{\bl}{{\bullet}}

\begin{document}

\hrule width0pt
\vsk->

\title[Solutions  modulo  $p$ of Gauss-Manin differential equations]
{Solutions  modulo  $p$ of Gauss-Manin differential equations for 
multidimensional hypergeometric integrals and associated Bethe ansatz }

\author
[ Alexander Varchenko]
{Alexander Varchenko$\>^\star$}

\maketitle

\begin{center}
{\it $^{\star}\<$Department of Mathematics, University
of North Carolina at Chapel Hill\\ Chapel Hill, NC 27599-3250, USA\/}

\end{center}

{\let\thefootnote\relax
\footnotetext{\vsk-.8>\noindent
$^\star\<${\sl E\>-mail}:\enspace anv@email.unc.edu\>,
supported in part by NSF grants DMS-1362924, DMS-1665239}}

\begin{abstract}
We consider the Gauss-Manin differential equations for hypergeometric integrals associated with a family of weighted arrangements of hyperplanes moving parallelly to themselves. We reduce these equations modulo a prime integer $p$ and construct polynomial solutions of the new differential equations as $p$-analogs of the initial hypergeometric integrals.

In some cases we interpret   the $p$-analogs of the hypergeometric integrals as sums over points of   
hypersurfaces defined over the finite field  $\FFF_p$.  That interpretation is similar to the interpretation
by Yu.I.\,Manin in \cite{Ma} of the number of points on an elliptic curve depending on a parameter
as a solution of a Gauss hypergeometric differential equation.

We discuss the associated Bethe ansatz.

\end{abstract}

{\small \tableofcontents  }

\setcounter{footnote}{0}
\renewcommand{\thefootnote}{\arabic{footnote}}

\section{Introduction}
Consider an arrangement of affine hyperplanes $(H_j)_{j=1}^n$ in $\C^k$. Let 
$f_j(t_1,\dots,t_k)$ be a first degree polynomial on $\C^k$ whose kernel is $H_j$.
Let $(a_j)_{j=1}^n$, $\ka$ be nonzero complex numbers. 
 An associated multidimensional hypergeometric integral is an integral of the form
\bea
I=\int_\ga {\prod}_{j=1}^n f_j(t_1,\dots,t_k)^{\frac{a_j}\ka} dt_1\wedge\dots\wedge dt_k,
\eea
where $\ga$ is a cycle in the complement to the union of the hyperplanes.  Assume that the hyperplanes depend on parameters $z_1,\dots,z_n$ and move parallelly to themselves
when the parameters change. Then the integral extends to a multivalued holomorphic function of the parameters. The holomorphic function is called a multidimensional hypergeometric function associated with this family of arrangements. The simplest example of such a function  is the classical hypergeometric function.

The multidimensional hypergeometric functions can be combined into collections so that the functions of a collection satisfy a system  of first order linear differential equations called the Gauss-Manin differential equations. 

If all polynomials $(f_j)_{j=1}^n$ have integer coefficients and the numbers $(a_j)$, $\ka$ are integers, then the Gauss-Manin differential equations can be reduced modulo a prime integer $p$ large enough.
The goal of this paper is to
construct polynomial solutions of the Gauss-Manin differential equations over the field $\FFF_p$ with $p$ elements. Our solutions are $p$-analogs of the multidimensional hypergeometric integrals.  The construction of the solutions
is motivated by  the classical  paper   \cite{Ma} by
Yu.I.\,Manin, cf.  Section ``Manin's Result: The Unity of Mathematics'' in \cite{Cl},
see also \cite{SV2, V5}.

\smallskip
The paper is organized as follows.  In Section \ref{Sec Arr} we remind the basic notions
associated with an affine arrangement of hyperplanes in $\C^k$.
In Section \ref{sec pt} we consider a family of arrangements of hyperplanes in $\C^k$
whose hyperplanes move parallelly to themselves when the parameters of the family  change. 
We introduce the Gauss-Manin differential equations and multidimensional hypergeometric integrals.
We show that the multidimensional hypergeometric integrals satisfy the Gauss-Manin differential equations, see Theorem \ref{thm GM}. In Section \ref{sec 4} we consider the reduction of this situation
modulo $p$ and construct polynomial solutions of the Gauss-Manin differential equations over $\FFF_p$,
see Theorem \ref{thm m},  which is
the main result of this paper. We interpret our solutions as integrals over
$\FFF_p^k$ under certain conditions, see Theorem \ref{thm 2int}.
 Such integrals could be considered as $p$-analogs of the multidimensional hypergeometric integrals.
In Section \ref{sec Ex} we consider examples. Under certain conditions we
interpret our polynomial solutions as  sums over points on some hypersurfaces over $\FFF_p$, see Theorem \ref{thm pts p3}. That statement is analogous to the interpretation in Manin\rq{}s  paper  \cite{Ma} 
of the number of points on an  elliptic curve depending on a parameter as a solution of a 
Gauss hypergeometric differential equation. In Section \ref{sec BA} 
we briefly discuss the associated Bethe ansatz. We 
introduce a system of the Bethe ansatz equations and construct a common eigenvector to geometric
Hamiltonians out of every solution  of the Bethe ansatz equations, see Theorem \ref{thm BAp}.
 We show that the Bethe eigenvectors
corresponding to distinct solutions are orthogonal with respect to the associated symmetric contravariant form, see Corollary \ref{cor orth}.

\section{Arrangements}
\label{Sec Arr}

We recall some facts about hyperplane arrangements, Orlik-Solomon algebras and 
flag complexes  from \cite{SV1}.

\subsection{An affine arrangement}
\label{An affine arrangement}
Let $k,n$ be positive integers, $k<n$. Denote $J=\{1,\dots,n\}$.

Let $\A =(H_j)_{j\in J}$,  be an arrangement of $n$ affine hyperplanes in
$\C^k$. Denote
$U = \C^k - \cup_{j\in J} H_j$,
the complement.
An {\it edge} $X_\al \subset \C^k$ of the arrangement $\A$ is a nonempty intersection of some
hyperplanes  of $\A$. Denote by
 $J_\al \subset J$ the subset of indices of all hyperplanes containing $X_\al$.
Denote
$l_\al = \mathrm{codim}_{\C^k} X_\al$.

We always assume that the arrangement $\A$ is {\it essential}, that is, $\A$ has a vertex, an edge which is a point.

An edge is called {\it dense} if the subarrangement of all hyperplanes containing it is irreducible:
the hyperplanes cannot be partitioned into nonempty sets so that, after a change of coordinates,
hyperplanes in different sets are in different coordinates. In particular, each hyperplane of $\A$ is
a dense edge.

\subsection{Flag complex}
For $\ell=0,\dots,k$, let $\on{Flag}^\ell(\A)$ denote the set of all flags
\bea
\C^k=L^0\supset L^1\supset\dots\supset L^\ell
\eea
with each $L^j$ an edge of $\A$ of codimension $j$. Let $\F^\ell(\A,\Z)$ denote the
quotient of the free abelian group on $\on{Flag}^\ell(\A)$ by the following relations. 
For every flag with a gap 
\bea \widehat F = (L^0 \supset L^1 \supset L^{i-1} \supset L^{i+1} \supset \dots \supset L^k),
\qquad i<k,
\eea 
we impose 
\bean 
\label{flagrel}
{\sum}_{F\supset \widehat F} F =0 
\eean
in $\F^\ell(\A,\Z)$, where the sum is over all flags $F = (\widetilde L^0 \supset \widetilde L^1 \supset \dots \widetilde \supset L^\ell) \in \on{Flag}^\ell(\A)$ such that $\widetilde L^j = L^j$ for all $j\neq i$.   The abelian group $\F^\ell(\A,\Z)$ is a free abelian group,
see \cite[Theorem 2.9.2]{SV1}.

There is an ``extension of flags'' differential $d: \F^\ell(\A, \Z) \to \F^{\ell+1}(\A,\Z)$ defined by 
\bea
 d(L^0 \supset L^1 \supset \dots \supset L^\ell) = 
{\sum}_{L^{\ell+1}} (L^0 \supset L^1 \supset \dots \supset L^\ell \supset L^{\ell+1}), 
\eea
where the sum is over all edges $L^{\ell+1}$ of $\A$ of codimension $\ell+1$ contained in $L^\ell$. It follows from \Ref{flagrel} that $d^2=0$. Thus we have a complex, the \emph{flag complex}, $(\F^\bl(\A,\Z), d)$.

\subsection{Orlik-Solomon algebra}
Define abelian groups $\OS^\ell(\A,\Z)$, $\ell=0,1,\dots,k$ as follows. For $\ell=0$, 
set $\OS^0(\A,\Z) = \Z$. For $\ell\geq 1$, $\OS^\ell(\A,\Z)$ is generated by 
$\ell$-tuples $(H_1,\dots,H_\ell)$ of hyperplanes $H_i\in \A$, subject to the relations:
\begin{enumerate}
\item[(i)] $(H_1,\dots,H_\ell) =0$ if $H_1,\dots, H_\ell$ are not in general position (i.e. if $\codim H_1 \cap \dots \cap H_\ell\neq \ell$);
\item[(ii)] $(H_{\sigma(1)},\dots, H_{\sigma(\ell)}) = (-1)^{|\sigma|} (H_1,\dots, H_\ell)$ for every permutation $\sigma \in \Sigma_\ell$;
\item[(iii)] for any $\ell+1$ hyperplanes $H_1,\dots,H_{\ell+1}$ that have non-empty intersection, $H_1\cap\dots\cap H_{\ell+1}\neq \emptyset$, and that are not in general position,
\bea
{\sum}_{i=1}^{\ell+1} (-1)^i (H_1,\dots,\widehat H_{i},\dots, H_{\ell+1}) =0,
\eea 
where $\widehat H_i$ denotes omission. 
\end{enumerate}
The abelian group $\OS^\ell(\A,\Z)$ is a free abelian group,
see  \cite{Bj}, \cite[Theorem 2.9.2]{SV1}.

The \emph{Orlik-Solomon algebra} of the arrangement $\A$ is the direct sum 
$\OS^\bl(\A, \Z)=$ $\oplus_{\ell=0}^k \OS^\ell(\A, \Z)$ endowed with the product given by 
$(H_1,\dots, H_i) \wedge (H'_1,\dots,H'_j) = (H_1,\dots,H_i,H'_1,\dots, H'_j)$. It is a graded skew-commutative algebra over $\Z$.

\subsection{Orlik-Solomon algebra as an algebra of differential forms}
\label{osdf}
For each hyperplane $H\in \A$, pick a polynomial $f_H$ of degree one on $\C^k$ whose zero set is $H$, i.e. let $f_H=0$ be an affine equation for $H$. Consider the logarithmic differential form 
\bea 
\iota(H):= d\log f_H = \frac{df_H}{f_H}
\eea 
on $\C^k$. Note that $\iota(H)$ does not depend on the choice of $f_H$ but only on $H$. Let $\bar\OS^\bl(\A,\Z)$ be the $\Z$-algebra of differential forms generated by $1$ and $\iota(H)$, $H\in \C$. The assignment $H\mapsto \iota(H)$ defines an isomorphism $\OS^\bullet(\A,\Z)\xrightarrow\sim \bar{\OS}^\bullet(\A,\Z)$ of graded algebras. 
\emph{Henceforth we shall not distinguish between $\OS$ and $\bar\OS$.}

\subsection{Duality, see \cite{SV1}, cf. \cite[Section 2.5]{VY}}
The vector spaces $\OS^\ell(\A,\Z)$ and $\FF^\ell(\A,\Z)$ are dual.
The pairing $ \OS^\ell(\A,\Z)\otimes\FF^\ell(\A,\Z) \to \Z$ is defined as follows.
{}For $H_{j_1},...,H_{j_\ell}$ in general position, set
$F(H_{j_1},...,H_{j_\ell})=(\tilde L^0\supset\dots\supset \tilde L^\ell)\in \F^\ell(\A,\Z)$,
where $\tilde L^0=\C^k$ and
$$
\tilde L^i = H_{j_1}\cap\dots\cap H_{j_i},\qquad
i=1,\dots,\ell.
$$
For any $F=(L^0\supset\dots\supset L^\ell) \in \FF^\ell(\A,\Z)$
 define $\langle (H_{j_1},...,H_{j_\ell}), F
 \rangle = (-1)^{|\sigma|},$
if $F
= F(H_{j_{\sigma(1)}},...,H_{j_{\sigma(\ell)}})$ for some $\sigma \in S_\ell$,
and $\langle (H_{j_1},...,H_{j_\ell}), F \rangle = 0$ otherwise.

\subsection{Flag  and Orlik-Solomon spaces over a field $\FFF$} For any field $\FFF$ and $\ell=0,\dots,k$ we define
\bean
\label{OS F}
\F^\ell(\A,\FFF) = \F^\ell(\A,\Z)\ox_\Z \FFF,
\qquad
\OS^\ell(\A,\FFF) = \OS^\ell(\A,\Z)\ox_\Z \FFF.
\eean

\subsection{Weights}
\label{sec weights}
An arrangement $\A$ is {\it weighted} if a map $a: J\to \C^\times$, 
$j\mapsto a_j,$ is given;
 $a_j$ is called the {\it weight} of $H_j$.
For an edge $X_\al$, define its weight as
$a_\al = \sum_{j\in J_\al}a_j$.

\subsection{Contravariant form and map, see \cite{SV1}}
The weights  determine a symmetric bilinear form $S^{(a)}$ on $\F^\ell(\A,\C)$, given by
\bea
S^{(a)}(F_1,F_2) =
{\sum}_{\{j_1, \dots , j_p\} \subset J} \ a_{j_1} \cdots a_{j_\ell}
\ \langle (H_{j_1}, \dots , H_{j_\ell}), F_1 \rangle
\langle (H_{j_1}, \dots , H_{j_\ell}), F_2 \rangle \ ,
\eea
where the sum is over all unordered $\ell$-element subsets.
The form is called the  {\it contravariant form}. It defines a homomorphism
 $$
\mathcal S^{(a)} : \FF^\ell(\A,\C) \to \FF^\ell(\A,\C)^*\simeq \OS^\ell(\A,\C),
\quad
 (L^0\supset\dots\supset L^\ell)\ \mapsto \
\sum \ a_{j_1} \cdots a_{j_\ell}\ (H_{j_1}, \dots , H_{j_\ell}),
$$
 where the sum is taken over all $\ell$-tuples $(H_{j_1},...,H_{j_\ell})$ such that
$$
H_{j_1} \supset L^1,\ {}\ .\ .\ .\ {} , \ {} H_{j_\ell}\supset L^\ell .
$$

\begin{thm} 
[{\cite[Theorem 3.7]{SV1}}]
For $\ell=1,\dots,k$, choose a basis of the free abelian group $\F^\ell(\A,\Z)$. Then with respect to that basis
the determinant of the contravariant form $S^{(a)}$ on $\F^\ell(\A,\Z)$ equals the product of suitable nonnegative integer
powers of the weights of all dense edges of $\A$ of codimension $\leq \ell$.

\end{thm}

\begin{cor}
\label{cor iso}
If the weights of all dense edges of $\A$ are nonzero, then the contravariant map
$\mc S^{(a)}: \F^\ell(\A,\C)\to \OS^\ell(\A,\C)$ is an isomorphism for all $\ell$.

\end{cor}

\subsection{Aomoto complex}

 Define
 \bean
\label{nu(a)}
\nu(a) = {\sum}_{j\in J} a_j (H_j)\ {}\ \in \ {}\ \OS^1(\A,\C)\ .
\eean
 Multiplication by $\nu(a)$ defines a differential
 $$
d^{(a)}\ :\   \OS^\ell(\A,\C)\ \to\ \OS^{\ell+1}(\A,\C) ,
\qquad
 x \ \mapsto\ x\wedge \nu(a) ,
$$
on  $\OS^\bullet(\A,\C)$, \ $(d^{(a)})^2=0$. The complex $(\OS^\bullet(\A,\C), d^{(a)})$
is called the {\it Aomoto complex}. 
The {\it master function} corresponding to the weighted arrangement $(\A, a)$ is the function
\bean
\label{mf1}
\Phi=\Phi_{\A,a}={\prod}_{j\in J} f_{H_j}^{a_j},
\eean
where each $f_{H_j}=0$ is an affine equation for the hyperplane $H_j$. Then $\nu(a) = d\Phi/\Phi$.

\begin{thm} [{\cite[Lemma 3.2.5]{SV1} and \cite[Lemma 5.1]{FMTV}}]
The Shapovalov map is a homomorphism of complexes
\bea
\mc S^{(a)} : (\F^\bullet(\A,\Z),d) \to (\OS^\bullet(\A,\Z),d^{(a)}).
\eea

\end{thm}

\subsection{Singular vectors}
An element $v \in \FF^k(\A,\C)$ is called  {\it singular}  if
 $\langle  d^{(a)} \OS^{\ell-1}(\A,\C), v\rangle =0$.
Denote by
$$
\Sing \FF^k(\A,\C) \subset \FF^k(\A,\C)
$$
the subspace of all singular vectors.

\subsection{Arrangements with normal crossings only}
\label{sec nco}

 An essential arrangement $\A$ is with {\it normal crossings} only,
if exactly $k$ hyperplanes meet at every vertex of $\A$.
Assume that $\A$ is an essential arrangement with normal crossings only.

A subset $\{j_1,\dots,j_\ell\}\subset J$ is called {\it  independent} if the hyperplanes
$H_{j_1},\dots,H_{j_\ell}$ intersect transversally.
A basis of $\OS^\ell(\A,\C)$ is formed by
$(H_{j_1},\dots,H_{j_\ell})$ where
$\{{j_1} <\dots <{j_\ell}\}$  are independent  $\ell$-element subsets of
$J$. The dual basis of $\FF^\ell(\A,\C)$ is formed by the corresponding vectors
$F(H_{j_1},\dots,H_{j_\ell})$.
These bases of $\OS^\ell(\A,\C)$ and $\FF^\ell(\A,\C)$ are called {\it standard}.

In $\FF^\ell(\A,\C)$ we  have
\bean
\label{skew}
F(H_{j_1},\dots,H_{j_\ell}) = (-1)^{|\sigma|}
F(H_{j_{\sigma(1)}},\dots,H_{j_{\sigma(\ell)}})
\eean
for any permutation  $\sigma \in S_\ell$.
For an independent subset $\{j_1,\dots,j_\ell\}$, we have
\\
$
S^{(a)}(F(H_{j_1},\dots,H_{j_\ell}) , F(H_{j_1},\dots,H_{j_\ell})) = a_{j_1}\cdots a_{j_\ell}
$
and
$
S^{(a)}(F(H_{j_1},\dots,H_{j_\ell})$,
\\
 $F(H_{i_1},\dots,H_{i_\ell})) = 0
$
for any distinct elements of the standard basis.

\section{A family of parallelly transported hyperplanes}
\label{sec pt}

\subsection{An arrangement in  $\C^n\times\C^k$}
\label{Ana}
Recall that $J=\{1,\dots,n\}$.
Consider $\C^k$ with coordinates $t_1,\dots,t_k$,\
$\C^n$ with coordinates $z_1,\dots,z_n$, the projection
$\C^n\times\C^k \to \C^n$.
Fix $n$ nonzero linear functions on $\C^k$,
$g_j = b_j^1t_1+\dots + b_j^kt_k$, $j\in J$,
where $b_j^i\in \C$.
Define $n$ linear functions on $\C^n\times\C^k$,
\bean
\label{bij}
f_j = z_j+g_j = z_j + b_j^1t_1+\dots + b_j^kt_k,\qquad j\in J.
\eean
In $\C^n\times \C^k$ define
 the arrangement
\bea
\tilde \A = \{ \tilde H_j\ | \ f_j = 0, \ j\in J \} .
\eea
Denote $\tilde U = \C^n\times \C^k - \cup_{j\in J} \tilde H_j$.

For every fixed $z^0=(z_1^0,\dots,z_n^0)$ the arrangement $\tilde \A$
induces an arrangement $\A(z^0)$ in the fiber over $z^0$ of the projection. We
identify every fiber with $\C^k$. Then $\A(z^0)$ consists of
hyperplanes $H_j(z^0), j\in J$, defined in $\C^k$ by the same equations
$f_j=0$. Denote
\bean
\label{U(A(z))}
U(\A(z^0)) = \C^k - \cup_{j\in J} H_j(z^0),
\eean
the complement to the arrangement $\A(z^0)$.

We assume that for any $z^0$ the arrangement $\A(z^0)$ has a vertex. This means that
the span of $(g_j)_{j\in J}$ is $k$-dimensional.

A point $z^0\in\C^n$ is called  {\it good} if $\A(z^0)$ has normal
crossings only.  Good points form the complement in $\C^n$ to the union
of suitable hyperplanes called the {\it discriminant}.

\subsection{Discriminant}
\label{Discr}

The collection $(g_j)_{j\in J}$ induces a
matroid structure  $\mc M_\C$ on $J$.  A subset $C=\{i_1,\dots,i_r\}\subset J$ is
a {\it circuit} in $\mc M_\C$  if $(g_i)_{i\in C}$ are linearly dependent but any
proper subset of $C$ gives linearly independent $g_i$'s.

For a circuit $C=\{i_1,\dots,i_r\}$, \  let
$(\la_C^i)_{i\in C}$ be a nonzero collection of complex numbers such that
$\sum_{i\in C}
\la_C^ig_i = 0$. Such a collection  is unique up to
multiplication by a nonzero number.

For every circuit $C$ we fix such a collection
and denote $f_C = \sum_{i\in C} \la_C^iz_i$.
The equation $f_C=0$ defines a hyperplane $H_C$ in
$\C^n$.
It is convenient to assume that $\la_C^i=0$ for $i\in J-C$ and write
$f_C = \sum_{i\in J} \la_C^iz_i$.

For any $z^0\in\C^n$, the hyperplanes $(H_i(z^0))_{i\in C}$ in $\C^k$ have nonempty
intersection if and only if $z^0\in H_C$. If $z^0\in H_C$, then the
intersection has codimension $r-1$ in $\C^k$.

Denote by $\frak C_\C$ the set of all circuits in $\mc M_\C$.
Denote  $\Delta = \cup_{C\in \frak C_\C} H_C$.
The arrangement $\A(z^0)$ in $\C^k$
has normal crossings only, if and only if $z^0\in \C^n-\Delta$.

\subsection{Good fibers}
\label{sec Good fibers}

For any
$z^1, z^2\in \C^n-\Delta$, the spaces $\FF^\ell(\A(z^1),\C)$, $\FF^\ell(\A(z^2),\C)$
 are canonically identified. Namely, a vector $F(H_{j_1}(z^1),\dots,H_{j_\ell}(z^1))$ of the first space
is identified  with
the vector $F(H_{j_1}(z^2),\dots,H_{j_\ell}(z^2))$ of the second.

Assume that weights $a=(a_j)_{j\in J}$ are given. Then each
arrangement $\A(z)$  is weighted.
The identification of spaces $\FF^\ell(\A(z^1),\C)$,
$\FF^\ell(\A(z^2),\C)$ for $z^1,z^2\in\C^n-\Delta$ identifies the corresponding subspaces
$\Sing\FF^k(\A(z^1),\C)$, $\Sing\FF^k(\A(z^2),\C)$ and 
the corresponding contravariant forms.

For a point $z^0\in\C^n-\Delta$, denote $V_\C=\FF^k(\A(z^0),\C)$, 
$\Sing V_\C=\Sing\FF^k(\A(z^0),\C)$.
The triple $(V_\C, \Sing V_\C, S^{(a)})$ does not depend on  $z^0\in\C^n-\Delta$
under the above identification.

\subsection{Geometric Hamiltonians, cf. \cite{V3, V4}}
\label{sec key identity}

For any circuit $C=\{i_1, \dots, i_r\}$, we 
define a linear operator
$L_C : V_\C\to V_\C$ in terms of the standard basis of
 $V_\C$, see Section
\ref{sec nco}.

For $m=1,\dots,r$, denote $C_m=C-\{i_m\}$.
Let $\{{j_1}<\dots <{j_k}\}
\subset J$ be an independent ordered subset and
$F(H_{j_1},\dots,H_{j_k})$
the corresponding element of the standard basis.
Define $L_C : F(H_{j_1},\dots,H_{j_k}) \mapsto 0$ if
$|\{{j_1},\dots,{j_k}\}\cap C| < r-1$.
If $\{{j_1},\dots,{j_k}\}\cap C = C_m$ for some
$1\leq m\leq r$, then using the skew-symmetry property \Ref{skew}
we can write
\bea
F(H_{j_1},\dots,H_{j_k})
\,=\,
\pm\, F(H_{i_1},H_{i_2},\dots,\widehat{H_{i_{m}}},\dots,H_{i_{r-1}}H_{i_{r}},H_{s_1},\dots,H_{s_{k-r+1}})
\eea
with $\{{s_1},\dots,{s_{k-r+1}}\}=
\{{j_1},\dots,{j_k}\}-C_m$.
Define
\bea
L_C
&:&
F(H_{i_1},\dots,\widehat{H_{i_{m}}},\dots,H_{i_{r}},H_{s_1},\dots,H_{s_{k-r+1}})
 \mapsto
\\
&&
\phantom{aaaaaaaaa}
(-1)^m {\sum}_{l=1}^{r} (-1)^l a_{i_l}
F(H_{i_1},\dots,\widehat{H_{i_{l}}},\dots,H_{i_{r}},H_{s_1},\dots,H_{s_{k-r+1}}) .
\eea

\begin{lem} [\cite{V3}]
\label{lem L_C}
The operator $L_C$ is symmetric with respect to the
contravariant form.
\end{lem}

Consider  the logarithmic differential 1-forms
\bea
\omega_j = \frac {df_j}{f_j}, \ j\in J,
\qquad
\omega_C = \frac {df_C}{f_C}, \ C\in \frak C_\C
\eea
in variables $t_1,\dots,t_k,z_1,\dots,z_n$.
For any circuit $C=\{i_1,\dots,i_r\}$, we have
$$
\omega_{i_1} \wedge \dots \wedge \omega_{i_{r}} =
\omega_C \wedge {\sum}_{l=1}^{r} (-1)^{l-1}
\omega_{i_1} \wedge \dots \wedge \widehat{\omega_{i_{l}}} \wedge \dots \wedge
\omega_{i_{r}} .
$$

\begin{lem} [{\cite[Lemma 4.2]{V3}, \cite[Lemma 5.4]{V4}}] 
\label{lem 2}
We have
\bean
\label{form H}
&&
{\sum}_{{\rm independent } \atop \{j_1 < \dots < j_k\} \subset J }
\Big( {\sum}_{j\in J} a_j
 \omega_j  \Big) \wedge \omega_{j_1} \wedge \dots \wedge \omega_{j_k}
\otimes
F(H_{j_1}, \dots , H_{j_k}) =
\\
&&
\phantom{aaaa}
{\sum}_{{\rm independent } \atop \{j_1 < \dots < j_k\} \subset J }
{\sum}_{C\in \frak C_\C}
\omega_{C} \wedge
 \omega_{j_1} \wedge \dots \wedge \omega_{j_k}
\otimes
L_C
F(H_{j_1}, \dots , H_{j_k}) .
\notag
\eean
\end{lem}

\begin{proof}
The lemma is a direct corollary of the definition of the maps $L_C$.
\end{proof}

The identity in \Ref{form H} is called the {\it key identity}.

\smallskip

Recall that $\omega_C = df_C/f_C$ and $f_C = \sum_{j\in J}\la_C^iz_i$. 
For $i\in J$, introduce the $\End_\C(V_\C)$-valued rational functions in $z_1,\dots,z_n$
by the formula
\bean
\label{K_j}
K_i(z)  =  {\sum}_{C\in \frak C_\C}
\,\frac{\la^i_C}{f_C(z)} \,L_C \,,
\qquad
i\in J .
\eean
Then
\bean
\label{LK}
{\sum}_{C\in \frak C_\C}
\omega_{C} \otimes L_C = {\sum}_{i\in J} dz_i\otimes K_i(z) .
\eean
The functions $K_i(z)$ are called  {\it geometric Hamiltonians}.

\begin{cor}
The geometric  Hamiltonians are symmetric with respect to the contravariant form,
$S^{(a)}(K_i(z)x, y)=S^{(a)}(x, K_i(z)y)$ for $i\in J$, $x,y\in V_\C$.

\end{cor}

\subsection{Gauss-Manin differential equations}
\label{sec GM}

The {\it Gauss-Manin differential equations} with parameter $\ka \in \C^\times$ is the following system of
differential equations on a $V_\C$-valued function $I(z_1,\dots,z_n)$,
\bean
\label{dif eqn}
\kappa \frac{\der I}{\der z_i}(z) = K_i(z)I(z),
\qquad
i\in J,
\eean
where  $K_i(z)$ are the geometric Hamiltonians defined in \Ref{K_j}.

Introduce the master function
\bean
\label{Mast}
\Phi(z,t,a) = {\prod}_{j\in J} f_j(z,t)^{a_j}
 \eean
 on $\tilde U\subset \C^n\times \C^k$.
The function $\Phi(z,t,a)^{1/\ka}$
defines a rank one local system $\mc L_\kappa$
on $\tilde U$, whose horizontal sections
over open subsets of $\tilde U$
 are univalued branches of $\Phi(z,t,a)^{1/\ka}$ multiplied by complex numbers.

For $z^0\in\C^n-\Delta$ and an element $\gamma(z^0)\in H_k(U(\A(z^0)), \mc L_\kappa\vert_{U(\A(z^0))})$,
we interpret the integration  map
$\OS^k(\A(z^0),\C)=V^* \to \C$,
$\om \mapsto \int_{\gamma(z^0)} \Phi(z^0,t,a)^{1/\ka}\om$
as an element of $\FF^k(\A(z^0),\C)=V_\C$.
The vector bundle
\bea
\cup_{z^0\in \C^n-\Delta}\,H_k(U(\A(z^0)), \mc L_\kappa\vert_{U(\A(z^0))}) \to
 {}  \C^n-\Delta
 \eea
has a canonical  flat Gauss-Manin connection. A locally constant section
 $\ga: z\mapsto \gamma(z) \in H_k(U(\A(z)), \mc L_\kappa\vert_{U(\A(z))})$  of the Gauss-Manin connection
defines a $V_\C$-valued function 
\bean
\label{def I}
\phantom{aaa}
 I^{(\ga)}(z_1,\dots,z_n)
=\sum_{{\rm independent } \atop \{j_1 < \dots < j_k\} \subset J }
\left( \int_{\gamma(z)} \Phi(z,t,a)^{1/\ka}  \omega_{j_1} \wedge \dots \wedge \omega_{j_k}\right)
F(H_{j_1},\dots,H_{j_k}) .
\eean
The integrals
\bea
 I^{(\ga)}_{j_1,\dots,j_k}(z_1,\dots,z_n)
= \int_{\gamma(z)} \Phi(z,t,a)^{1/\ka} \omega_{j_1} \wedge \dots \wedge \omega_{j_k}
\eea
are called the {\it
multidimensional hypergeometric integrals} associated with the master function $\Phi(z,t,a)$.

\begin{thm} [\cite{V4}]
\label{thm GM}
The function $I^{(\ga)}$ takes values in $\Sing V_\C$ and is a solutions of the Gauss-Manin differential equations.

\end{thm}

The condition that the function $I^{(\ga)}$ takes values in $\Sing V_\C$ may be reformulated as the system of equations
\bean
\label{sieq}
{\sum}_{j\in J} a_j  I^{(\ga)}_{j, j_2,\dots,j_k} = 0,
\qquad  \on{for}\
j_2,\dots,j_k\in J.
\eean

\subsection{Proof of Theorem \ref{thm GM}}

We sketch the proof following \cite{SV1}, cf. \cite{SV2}.
 The intermediate statements of this sketch will be used later when 
constructing solutions of the Gauss-Manin differential  equations over a finite field
$\FFF_p$.
 The proof  of Theorem \ref{thm GM} is based on the following cohomological relations
\Ref{id111} and \Ref{der zi}.

For any $j_1,\dots,j_k\in J$ denote
\bean
\label{detd}
d_{j_1,\dots,j_k} = {\det}_{i,l=1}^k(b^i_{j_l}),
\qquad
W_{j_1,\dots,j_k}(z,t) = \frac{d_{j_1,\dots,j_k}}{\prod_{l=1}^kf_{j_l}(z,t)}.
\eean
We have
\bean
\label{tt-part}
\om_{j_1}\wedge\dots\wedge\om_{j_k} = W_{j_1,\dots,j_k}(z,t)\, dt_1\wedge\dots\wedge dt_k + \dots,
\eean
where the dots denote the terms having differentials  $dz_{j_1}, \dots, dz_{j_k}$. Notice that the rational function  $ W_{j_1,\dots,j_k}(z,t)$ has the form
\bean
\label{rat0}
P_{j_1,\dots,j_k}(z,t)
{\prod}_{j\in J} f_j(z,t) ^{-1},
\eean
where $P_{j_1,\dots,j_k}(z,t)$  is polynomial with integer coefficients in variable $z_1,\dots,z_n, t_1,\dots,t_k$ 
and $b^i_j$, $j\in J$, $i=1,\dots,n$, see \Ref{bij}.
For any $j_2,\dots,j_k\in J$ we write
\bean
\label{ttt-part}
\om_{j_2}\wedge\dots\wedge\om_{j_k} = {\sum}_{l=1}^k W_{j_2,\dots,j_k; l}(z,t)\, dt_1\wedge\dots\wedge
\widehat{dt_l}\wedge\dots \wedge dt_k + \dots,
\eean
where the dots denote the terms having differentials  $dz_{j_2},\dots,dz_{j_k}$ and
$ W_{j_1,\dots,j_k;l}$ are rational functions in $z,t$ of the form
\bean
\label{rat}
P_{j_1,\dots,j_k;l}(z,t)
{\prod}_{j\in J} f_j(z,t) ^{-1},
\eean
where $P_{j_1,\dots,j_k;l}(z,t)$  are polynomials with integer coefficients in variable $z_1,\dots,z_n, t_1,\dots,t_k$ 
and $b^i_j$, $j\in J$, $i=1,\dots,n$, see \Ref{bij}.
The formula
\bean
\label{id1}
\nu(a) \wedge \om_{j_2} \wedge\dots\wedge \om_{j_k}
={\sum}_{j\in J} a_j \,\om_j
\wedge \om_{j_2} \wedge\dots\wedge \om_{j_k}
\eean
implies the identity
\bean
\label{id111}
&&
\ka \, d_t\Big( \Phi(z,t,a)^{1/\ka} {\sum}_{l=1}^k W_{j_2,\dots,j_k;l}(z,t)\, dt_1\wedge\dots\wedge
\widehat{dt_l}\wedge\dots dt_k  \Big)
\\
\notag
&&
\phantom{aaaaaa}
=  {\sum}_{j\in J} a_j 
\Phi(z,t,a)^{1/\ka} W_{j, j_2,\dots,j_k}(z,t)\,  dt_1\wedge\dots\wedge dt_k,
\eean
where $d_t$ denotes the differential with respect to the variables $t$.

Now we deduce a corollary of the key identity  \Ref{form H}. 
Choose $i\in J$. For any  independent $ \{j_1 < \dots < j_k\} \subset J$, we write
\bean
\label{dzi}
&&
\Phi(z,t,a)^{1/\ka} \om_{j_1} \wedge\dots\wedge \om_{j_k}
 = \Phi(z,t,a)^{1/\ka} W_{j_1,\dots,j_k}(z,t)\, dt_1\wedge\dots\wedge dt_k
\\
\notag
&&\phantom{aaa}
+
dz_i\wedge
\Big(\Phi(z,t,a)^{1/\ka}{\sum}_{l=1}^k W_{j_1,\dots,j_k;i,l}(z,t)\, dt_1\wedge\dots\wedge
\widehat{dt_l}\wedge\dots\wedge dt_k \Big)
+ \dots,
\eean
where the dots denote the terms which contain $dz_j$ with $j\ne i$, 
and the coefficients \\
$W_{j_1,\dots,j_k;i,l}(z,t)$ are rational functions in $z,t$ of the form
\bean
\label{rat2}
P_{j_1,\dots,j_k;i,l}(t,z)
{\prod}_{j\in J} f_j(z,t)^{-1},
\eean
where $P_{j_1,\dots,j_k;i,l}(z,t)$
 are polynomials with integer coefficients in variable $z_1,\dots,z_n, t_1,\dots,t_k$ 
and $b^i_j$, $j\in J$, $i=1,\dots,n$, see \Ref{bij}.

Formula  \Ref{form H} implies  that for any $i\in J$ we have
\bean
\label{der zi}
&&
\phantom{aa}
\ka{\sum}_{{\rm independent } \atop \{j_1 < \dots < j_k\} \subset J }
\Big(
 \frac{\der}{\der z_i}
 \Big(
\Phi(z,t,a)^{1/\ka}W_{j_1,\dots,j_k}(z,t)  \Big) dt_1\wedge\dots\wedge dt_k
\\
\notag
&&
\phantom{aa}
 +
d_t\Big(\Phi(z,t,a)^{1/\ka}{\sum}_{l=1}^n W_{j_1,\dots,j_k;i,l}(t,z)\, dt_1\wedge\dots\wedge
\widehat{dt_l}\wedge\dots\wedge dt_k\Big)
 \Big) F(H_{j_1},\dots,H_{j_k})
\\
\notag
&&
\phantom{aaa}
     =  K_i(z){\sum}_{{\rm independent } \atop \{j_1 < \dots < j_k\} \subset J }
 \Phi(z,t,a)^{1/\ka} W_{j_1,\dots,j_k}(z,t) \,dt_1\wedge\dots\wedge dt_k\,  F(H_{j_1},\dots,H_{j_k}), 
\eean
where $d_t$ denotes the differential with respect to the variables $t$.

Integrating both sides of equations  \Ref{id111} and \Ref{der zi} over $\ga(z)$ and using Stokes' theorem we obtain equations
\Ref{sieq} and \Ref{dif eqn} for the vector $I^{(\ga)}(z)$ in \Ref{def I}. Theorem \ref{thm GM} is proved.

\subsection{Remarks}

It is known from \cite{SV1} that for generic $\ka$ all $\Sing V_\C$-valued solutions 
of the Gauss-Manin equations \Ref{dif eqn} are given by formula \Ref{def I}. Hence we have the following statement.

\begin{thm} [\cite{V4}]

\label{thm fl}
The geometric Hamiltonians $H_i(z)$, $i\in J$ preserve $\Sing V_\C$ and commute on $\Sing V_\C$, namely,
$[H_i(z^0)\big\vert_{\Sing V_\C},H_j(z^0)\big\vert_{\Sing V_\C}] =0$ for all $i,j\in J$ and $z^0\in\C^n-\Delta$.

\end{thm}

\section{Reduction modulo $p$ of a family  of parallelly transported hyperplanes}
\label{sec 4}

\subsection{An arrangement in  $\C^n\times\C^k$ over $\Z$}
\label{AnaZ} 
Similarly to Section \ref{Ana} consider
 $\C^k$ with coordinates $t_1,\dots,t_k$,\
$\C^n$ with coordinates $z_1,\dots,z_n$, the projection
$\C^n\times\C^k \to \C^n$.  Fix $n$ nonzero linear functions on $\C^k$,
$g_j = b_j^1t_1+\dots + b_j^kt_k$, $j\in J$,
{\it with integer coefficients} $b_j^i\in \Z$.
Define $n$ linear functions on $\C^n\times\C^k$,
\bean
\label{Zf}
f_j = z_j+g_j = z_j + b_j^1t_1+\dots + b_j^kt_k
\eean
 where $ j\in J$.

Recall the matroid structure $\mc M_\C$ on $J$, the set  $\frak C_\C$
of all circuits in $\mc M_\C$,
 and the  linear functions $f_C= \sum_{i\in J} \la_C^iz_i$ labeled by $C\in \frak C_\C$, where the functions are
defined in Section \ref{Discr}.  Each of these functions is determined
up to multiplication by a nonzero constant. 

\begin{defn}
\label{defn}
We fix the coefficients $(\la_C^i)_{i\in J}$ to be integers such that the greatest common divisor of
$(\la_C^i)_{i\in J}$ equals 1.
\end{defn}

This is possible since all $b^i_j$ are integers. 
This choice of the coefficients defines the function $f_C$ uniquely up to multiplication
by $\pm1$.

Let $p$ be a prime integer and $\FFF_p$ the field with $p$ elements. Let $[\,]:\Z\to \FFF_p$
be the natural projection. Introduce the following linear functions in $z,t$ with coefficients in $\FFF_p$,
\bean
\label{prg}
[g]_j: 
&=& 
{\sum}_{i=1}^k[b^i_j]t_i,
\qquad
[f]_j: =z_j+[g]_j,
\qquad
j\in J,
\\
\notag
{}[f]_C :
&=&
{\sum}_{i\in J} [\la_C^i]z_i,
\qquad C\in \frak C_\C.
\eean
The collection $([g]_j)_{j\in J}$ induces a
matroid structure  $\mc M_{\FFF_p}$ on $J$.  A subset $C=\{i_1,\dots,i_r\}\subset J$ is
a {\it circuit} in $\mc M_{\FFF_p}$  if $([g]_i)_{i\in C}$ are linearly dependent over $\FFF_p$
but any
proper subset of $C$ gives linearly independent $[g]_i$'s.

\begin{defn}
\label{def p}
 We say that a prime integer $p$ is {\it good
with respect to the collection of linear functions} $(g_j)_{j\in J}$ if all linear functions
in \Ref{prg} are nonzero and the matroid structures $\mc M_\C$ and $\mc M_{\FFF_p}$ on $J$ are the same.
\end{defn}

In what follows we always assume that $p$ is good with respect to the collection of linear functions $(g_j)_{j\in J}$.

We have logarithmic differential forms
\bea
[\omega]_j = \frac {d[f]_j}{[f]_j}, \ j\in J,
\qquad
[\omega]_C = \frac {d[f]_C}{[f]_C}, \ C\in \frak C_{\FFF_p} = \frak C_\C
\eea
in variables $t_1,\dots,t_k,z_1,\dots,z_n$ with coefficients in $\FFF_p$.
For any circuit $C=\{i_1,\dots,i_r\}$, we have
$$
[\omega]_{i_1} \wedge \dots \wedge [\omega]_{i_{r}} =
[\omega]_C \wedge {\sum}_{l=1}^{r} (-1)^{l-1}
[\omega]_{i_1} \wedge \dots \wedge \widehat{[\omega]_{i_{l}}} \wedge \dots \wedge
[\omega]_{i_{r}} .
$$

Assume that the nonzero integer weights $a=(a_j)_{j\in J}$ are given,
$a_j\in\Z$, $a_j\ne 0$.  The constructions of Section \ref{sec pt}
give us 
\begin{enumerate}
\item[(i)]
a vector space $V_{\FFF_p}$   over $\FFF_p$ with  standard basis 
$(F(H_{j_1},\dots,H_{j_k}))$ indexed by 
all independent subsets   $\{j_1 < \dots < j_k\}$ of $J$;

\item[(ii)]
a vector subspace $\Sing V_{\FFF_p} \subset V_{\FFF_p}$
consisting of all linear combinations
\\
$\sum _{{\rm independent } \atop \{j_1 < \dots < j_k\} \subset J }
I_{j_1,\dots,j_k} F(H_{j_1},\dots,H_{j_k})$ satisfying the  equations
\bean
\label{sieqp}
{\sum}_{j\in J} [a_j]  I_{j, j_2,\dots,j_k} = 0,
\qquad  \on{for}\
j_2,\dots,j_k\in J;
\eean

\item[(iii)]
a symmetric bilinear ${\FFF_p}$-valued contravariant form 
$[S]^{(a)}$ on $V_{\FFF_p}$ defined by the formulas
\bea
[S]^{(a)}(F(H_{j_1},\dots,H_{j_k}) , F(H_{j_1},\dots,H_{j_k})) = [a_{j_1}]\cdots [a_{j_k}]
\eea
for any independent $\{j_1<\dots<j_k\}$ and
$
[S]^{(a)}(F(H_{j_1},\dots,H_{j_\ell})$,
 $F(H_{i_1},\dots,H_{i_\ell})) = 0$
for any distinct elements of the standard basis.

\end{enumerate}

For any circuit $C=\{i_1, \dots, i_r\}$, we 
define a linear operator
$[L]_C : V_{\FFF_p}\to V_{\FFF_p}$ by the formula of
Section \ref{sec key identity} in which the numbers $a_{i_l}$
are replaced with $[a_{i_l}]$. 
We have the key identity
\bean
\label{fHp}
&&
{\sum}_{{\rm independent } \atop \{j_1 < \dots < j_k\} \subset J }
\Big( {\sum}_{j\in J} [a_j]
 [\omega]_j  \Big) \wedge [\omega]_{j_1} \wedge \dots \wedge [\omega]_{j_k}
\otimes
F(H_{j_1}, \dots , H_{j_k}) =
\\
&&
\phantom{aaaa}
{\sum}_{{\rm independent } \atop \{j_1 < \dots < j_k\} \subset J }
{\sum}_{C\in \frak C_{\FFF_p}}
[\omega]_{C} \wedge
 [\omega]_{j_1} \wedge \dots \wedge [\omega]_{j_k}
\otimes
[L]_C
F(H_{j_1}, \dots , H_{j_k}) .
\notag
\eean
For $i\in J$, we define the $\End_{\FFF_p}(V_{\FFF_p})$-valued rational functions in $z_1,\dots,z_n$
by the formula
\bean
\label{Kjp}
[K]_i(z) \ = \ {\sum}_{C\in \frak C_{\FFF_p}}
\,\frac{[\la^i_C]}{[f]_C(z)} \,[L]_C \,,
\qquad
i\in J .
\eean
We call the functions $[K]_i(z)$   the {\it geometric Hamiltonians}.
The geometric  Hamiltonians are symmetric with respect to the contravariant form,
$[S^{(a)}]([K]_i(z)x, y)=[S]^{(a)}(x, [K]_i(z)y)$ for $i\in J$, $x,y\in V_{\FFF_p}$.

The {\it Gauss-Manin differential equations over $\FFF_p$} with parameter $[\ka] \in \FFF_p^\times$ is the following system of differential equations,
\bean
\label{dif eqn p}
[\kappa] \frac{\der I}{\der z_i}(z) = [K]_i(z)I(z),
\qquad
i\in J.
\eean

The goal of this paper is to construct polynomial $\Sing V_{\FFF_p}$-valued solutions of these differential equations.

\subsection{Polynomial solutions}
Let a prime integer $p$ be good with respect to  $(g_j)_{j\in J}$.
Let $a=(a_j)_{j\in J}$ be nonzero integer weights
$a_j\in\Z$, $a_j\ne 0$. 

Choose  positive integers $A=(A_1,\dots,A_n)$, such that
\bean
\label{Mjp}
[A_j]  = \frac{[a_j]}{[\ka]} 
\eean
in $\FFF_p$.
Introduce the {\it master polynomial}
\bean
\label{Phip}
\Phi(z,t,A) = {\prod}_{j\in J} f_j(z,t)^{A_j} \ \ \in \ \ \Z[z_1,\dots,z_n, t_1,\dots,t_k],
\eean
where $f_j(z,t)$ are defined in \Ref{Zf}.   For any $j_1,\dots,j_k$ the function
$X_{j_1,\dots,j_k}(z,t,A): = 
\\
\Phi(z,t,A) \frac{d_{j_1,\dots,j_k}}{{\prod}_{l=1}^kf_{j_l}(z,t)} $ is a polynomial in $z,t$ with integer coefficients.
For fixed $q=(q_1,\dots,q_k)$  $ \in \Z^k$  consider the Taylor expansion
\bean
\label{Te}
X_{j_1,\dots,j_k}(z,t,A) = {\sum}_{i_1,\dots,i_k\geq 0} I_{j_1,\dots,j_k}^{ i_1,\dots,i_k}(z,q,A)\,
(t_1-q_1)^{i_1}\dots (t_k-q_k)^{i_k},
\eean
where $ I_{j_1,\dots,j_k}^{ i_1,\dots,i_k}(z,q,A)\in \Z[z_1,\dots,z_n]$ for any $i_1,\dots,i_k$. We denote
by $[ I]_{j_1,\dots,j_k}^{ i_1,\dots,i_k}(z,q,A)$ the projection of $ I_{j_1,\dots,j_k}^{ i_1,\dots,i_k}(z,q,A)$
to $\FFF_p[z_1,\dots,z_n]$.  Denote
\bean
\label{psp}
&&
[I]^{i_1,\dots,i_k}(z,q,A) 
\\
\notag
&&
\phantom{aaaa}
= {\sum} _{{\rm independent } \atop \{j_1 < \dots < j_k\} \subset J }
[I]_{j_1,\dots,j_k}^{i_1,\dots,i_k}(z,q,A) \,F(H_{j_1},\dots,H_{j_k}) \quad \in \quad
V_{\FFF_p}\ox \FFF_p[z_1,\dots,z_n].
\eean

\begin{thm}
\label{thm m}

Let  a prime integer $p$ be good with respect to  $(g_j)_{j\in J}$.
Then for any integers $A=(A_1,\dots,A_n)$ satisfying \Ref{Mjp}, any integers
 $q=(q_1,\dots,q_k)$ and any positive integers $\l=(l_1,\dots,l_k)$, the polynomial
function $I(z)=[I]^{l_1p-1,\dots,l_kp-1}(z,q,A)$
satisfies the algebraic equations in \Ref{sieqp} 
and the Gauss-Manin differential equations in \Ref{dif eqn p}.

\end{thm}

The parameters $A$, $q$, $l_1p-1,\dots,l_kp-1$ of the
solution $I(z)$ are analogs of locally constant cycles $\ga(z)$ in Section \ref{sec GM}.

 Notice that the space of polynomial solutions
of the equations \Ref{sieqp} 
and  \Ref{dif eqn p}  is a module over the ring 
$\FFF_p[z_1^p,\dots,z_n^p]$ since $\frac{\der z_i^p}{\der z_j} =0$.

\begin{proof}

To prove that $I(z)$  satisfies \Ref{sieqp} and
\Ref{dif eqn p} consider the Taylor expansions at $t=q$ of both sides of equations
\Ref{id111} and 
\Ref{der zi} divided by $dt_1\wedge\dots\wedge dt_k$.
Notice that the Taylor expansions are well defined due to formulas \Ref{rat0}, \Ref{rat}, and \Ref{rat2}. 
We project the Taylor expansions to $V_{\FFF_p}\ox \FFF_p[z_1,\dots,z_n]$.
 Then  the terms coming from
the $d_t$-summands  equal 
zero since $d(t_i^{l_ip})/dt_i=l_ipt_i^{l_ip-1}\equiv 0$ (mod $p$).
\end{proof} 

\subsection{Relation of solutions to integrals over $\FFF_p^k$}

For a polynomial $F(t_1,\dots,t_k)$ $\in\FFF_p[t_1$, \dots, $t_k]$ and a subset $\ga\subset \FFF^k_p$
define the integral 
\bea
\int_{\ga} F(t_1,\dots,t_k)\, dt_1\wedge\dots\wedge dt_k := {\sum}_{(t_1,\dots, t_k)\in\ga}F(t_1,\dots,t_k).
\eea

Consider the vector of polynomials 
\bean
\label{Fv}
&&
\phantom{aaaa}
F(x,t,A)
\\
\notag
&&
\phantom{aaa}
 = \!\!\sum _{{\rm independent } \atop \{j_1 < \dots < j_k\} \subset J }\!\!
{\prod}_{j\in J} \big([f]_j(x,t)\big)^{A_j}  \frac{[d_{j_1,\dots,j_k}]}{{\prod}_{l=1}^k[f]_{j_l}(x,t)}
 \,F(H_{j_1},\dots,H_{j_k})  \in 
V_{\FFF_p}\ox \FFF_p[t_1,\dots,t_k].
\eean

\begin{thm}
\label{thm 2int}  Let $x=(x_1,\dots,x_n)\in \FFF_p$.  Let
$[I]^{p-1,\dots,p-1}(z,q,A)$ be the solution of equations 
\Ref{sieqp} and \Ref{dif eqn p} considered in Theorem \ref{thm m} for $(l_1,\dots,l_k)=(1,\dots,1)$.

 \begin{enumerate}
\item[(i)]
If $\deg_{t_i} F(x,t,A) < 2p-2$ for $i=1,\dots,k$.  Then 
\bean
\label{2INT}
[I]^{(p-1,\dots,p-1)}(x,q,A) =(-1)^k\int_{\FFF_p^k}F(x,t,A)\,dt_1\wedge\dots\wedge dt_k .
\eean
\item[(ii)]
If the integers $A=(A_j)_{j\in J}$ are such that
\bean
\label{ineq}
A_1+\dots+A_n-k < (k+1)(p-1),
\eean
then \Ref{2INT} holds.
\end{enumerate}

\end{thm}

\begin{proof}
Part 1 follows from the statement: for a positive integer $i$,
\bean
\label{sum=0}
 \sum_{t\in\FFF_p}t^i \ \on{equals}\  -1 \ \on{if}\  (p-1)\big| i\
\on{ and\ equals\ zero\ otherwise}.
\eean
Part 2 also follows from \Ref{sum=0} by the following reason.
The polynomial 
\\
$P = \prod_{j\in J} \big([f]_j(x,t)\big)^{A_j}  \frac{[d_{j_1,\dots,j_k}]}{\prod_{l=1}^k[f]_{j_l}(x,t)}$ is a product
of $A_1+\dots+A_n-k$ linear functions in $t$. If \Ref{ineq} holds, then 
 $(t_1\dots t_k)^{p-1}$ is the only  monomial  $t_1^{m_1}\dots t_k^{m_k}$ of the
Taylor expansion of that polynomial such that $m_l>0$ and $(p-1)\big|m_l$ for $l=1,\dots,k$.
\end{proof}

The integral in \Ref{2INT} is a $p$-analog of the  hypergeometric integral \Ref{def I}, see also Section \ref{sec Ex}.

\section{Examples}
\label{sec Ex}

\subsection{Case $k=1$, see \cite{SV2}}

Let $\ka$, $a=(a_1,\dots,a_n)$ be nonzero complex numbers. Consider the master function
of complex variables
\bean
\label{Phi1}
\Phi(t_1,z_1,\dots,z_n,a)={\prod}_{i=1}^n(t_1+z_i)^{a_i}.
\eean
Let $z^0=(z^0_1,\dots,z_n^0)\in\C^n$ be a vector with distinct coordinates.
Consider the  $n$-vector  $I^{(\ga)} (z^0)=(I_1(z^0),\dots,I_n(z^0))$, where
\bean
\label{s}
I_j=\int _{\ga(z^0)} \Phi(t_1,z_1^0,\dots,z_n^0,a)^{1/\ka} \frac {dt_1}{t_1+z^0_j},\qquad j=1,\dots,n.
\eean
The integrals are over a closed (Pochhammer) curve $\ga(z^0)$ in $\C-\{z_1^0,\dots,z_n^0\}$ on which one fixes a uni-valued branch 
of the master function to make the integral well-defined.
Starting from such a curve chosen for given $\{z_1^0,\dots,z_n^n\}\subset\C$, the vector $I^{(\ga)}(z^0)$ can be analytically continued as a multivalued holomorphic function of $z$ to the complement in $\C^n$ to the union of the
diagonal  hyperplanes $z_i=z_j$.

\begin{thm}

 The vector $I^{(\ga)}(z)$ satisfies the algebraic equation
\bean
\label{le}
a_1I_1(z)+\dots+a_nI_n(z)=0
\eean
and  the differential equations:
\bean
\label{KZ}
\ka \frac{\partial I}{\partial z_i} \ = \
    \sum_{j \ne i}
   \frac{\Omega_{i,j}}{z_i - z_j}  I ,
\qquad i = 1, \dots , n,
\eean
where
\[ \Omega_{i,j} \ = \ \begin{pmatrix}

                   & \vdots^i &  & \vdots^j &  \\

        {\scriptstyle i} \cdots & { -a_j  } & \cdots &
              a_j    & \cdots \\

                   & \vdots &  & \vdots &   \\

        {\scriptstyle j} \cdots & a_i & \cdots & -a_i &
                 \cdots \\

                   & \vdots &  & \vdots &

                   \end{pmatrix} ,
                    \]
and all the
remaining  entries equal  zero, see \cite{SV1}, \cite[Section 1.1]{V2}.
\end{thm}

\begin{exmp}
\label{ex1}
Let $\ka=2$,  $n=3$, $a_1=a_2=a_3=-1$. Then $I^{(\ga)}(z)=(I_1(z),I_2(z),I_3(z))$,  where
\bean
\label{I3}
I_j(z) = \int_{\ga(z)} \frac 1{\sqrt{(t+z_1)(t+z_2)(t+z_3)}}\frac {dt}{t+z_j}.
\eean
In this case, the curve $\ga(z)$ may be thought of as a closed path on the elliptic curve
\bea
y^2=(t+z_1)(t+z_2)(t+z_3).
\eea
Each of these integrals is an elliptic integral. Such an integral is
 a branch of analytic continuation of a suitable Euler hypergeometric function up to change of variables.

\end{exmp}

\begin{exmp}
\label{ex2}
Let  $p>3$
be a prime integer. 
Let  $\ka=2$, $a_1=\dots=a_n=-1$, cf. Example \ref{ex1}.   For such $\ka$ and $a_j$,
 the algebraic equations
\Ref{le} and the differential equations \Ref{KZ} are well-defined when reduced modulo $p$. 
Choose the master polynomial
\bean
\Phi(t_1,z_1,\dots,z_n) = {\prod}_{i=1}^n(t_1+z_i)^{\frac{p-1}2}.
\eean
Consider the Taylor expansion of the polynomial 
\bean
\label{te}
F(t_1,z)= {\prod}_{i=1}^n(t_1+z_i)^{\frac{p-1}2} \Big(\frac 1{t_1+z_1}, \dots, \frac 1{t_1+z_n}\Big)
= {\sum}_i  I^i(z) t_1^i,
\eean
 see \Ref{Fv}. Let $[I]^i(z)$ be the projection of $I^i(z)$ to $(\FFF_p[z])^n$.
Then the vector $I(z):=[I]^{p-1}(z)$
is a solution of the differential equations  \Ref{KZ} over $\FFF_p[z]$ 
and $I_1(z)+\dots+I_n(z)=0$, see Theorem \ref{thm m}.

If $n\leq 4$, and $x=(x_1,\dots,x_n)\in \FFF_p^n$, then
\bea
I(x) = \int_{\FFF_p} F(t_1, x)\, dt_1
\eea
by Theorem \ref{thm 2int}.

Let $x=(x_1,\dots,x_n)\in\FFF_p^k$. Let 
$\Ga(x)$ be
 the affine curve
\bean
\label{ec}
y^2 = (t_1+x_1)\dots(t_1+x_n)
\eean
over $\FFF_p$. For a rational function $h : \Ga(x) \to \FFF_p$  define the integral  
\bean
\label{int}
\int_{\Ga(x)} h = {\sum}_{P\in\Ga(x)}' h(P),
\eean
as the sum over all points $P\in\Ga(x)$, where $h(P)$  is defined.

\begin{thm}
\label{thm pts}

Let $n$ equal $3$ or $4$. 
Let $[I]^{p-1}(x)=( [I]^{p-1}_1(x)$, \dots,
$ [I]^{p-1}_n(x))$ 
be  the vector of polynomials  obtained from
\Ref{te}. Then 
\bean
\label{pts=sol} 
\int_{\Ga(x)}\frac 1{t_1+x_j}
& = &  - \, [I]^{p-1}_j(x),
  \qquad j=1,\dots,n.
\eean
\end{thm}

\begin{rem}
Theorems  \ref{thm m}  and \ref{thm pts}  say that the integrals 
$\int_{\Ga(x_1,\dots,x_n)}\frac 1{t_1+x_j}$
are polynomials in $x_1,\dots,x_n\in\FFF_p$ and the tuple of polynomials
\bea
 \left( \int_{\Ga(x_1,\dots,x_n)}\frac 1{t_1+x_1},\dots,  
\int_{\Ga(x_1,\dots,x_n)}\frac 1{t_1+x_n}\right)
\eea
in these discrete variables
satisfies the system of Gauss-Manin differential equations. Cf. Example \ref{ex1}.

\end{rem}

\begin{rem}
In \cite[Section 2]{Ma} and in \cite{Cl}, an equation analogous to
\Ref{pts=sol}  for $n=3$  is considered, where the left-hand side is the number of points
on $\Ga(x_1,x_2,x_3)$ over $\FFF_p$ and the right-hand side is the reduction modulo $p$ of a solution of a 
second order Gauss hypergeometric differential equation. Notice that the number of points on $\Ga(x_1,x_2,x_3)$
 is the discrete integral over $\Ga(x_1,x_2,x_3)$ of the 
constant function $h=1$.

\end{rem}

{\it Proof of Theorem \ref{thm pts}.}
The proof is analogous to the reasoning in \cite[Section 2]{Ma} and \cite{Cl}. 
It is easy to see that 
\bea
&&
\int_{\Ga(x_1,\dots,x_n)}\frac 1{t_1+x_j}
 = 
{\sum}_{t_1\in\FFF_p,\, t_1\ne x_j} \frac 1{t_1+x_j} 
+{\sum}_{t_1\in\FFF_p} \frac {1}
{t_1+x_j}{\prod}_{s=1}^n(t_1+x_s)^{\frac{p-1}2}
\\
&&
\phantom{aa}
=  {\sum}_{t_1\in\FFF_p} (t_1+x_j)^{p-2} 
 + {\sum}_{t_1\in\FFF_p}  \sum_i  [I]^i_j(x_1,\dots,x_n) t_1^i
=   -  [I]^{p-1}_j(x_1,\dots,x_n). 
\eea
Notice that $  \sum_{t_1\in\FFF_p}  \sum_i  [I]^i_j(x_1,\dots,x_n) t_1^i
=   -  [I]^{p-1}_j(x_1,\dots,x_n)$ since for $n=3,4$ the degree of the left-hand side is less than
$2p-2$, see Theorem \ref{thm 2int}.
\qed

\end{exmp}

See
more examples with $k=1$  in \cite[Section 1]{SV2}.

\subsection{Counting on two-folded covers}
\label{cpww}

As in Section \ref{AnaZ}  consider $n$ nonzero linear functions on $\C^k$,
$g_j = b_j^1t_1+ \dots+b_j^kt_k$, $j\in J$, with integer coefficients $b_j^i\in \Z$.
Let a prime integer $p>3$ be good  with respect to $(g_j)_{j\in J}$. 

Assume that all weights $(a_j)_{j\in J}$ are equal to $-1$ and $\ka = 2$. Under these assumptions consider
the algebraic equations \Ref{sieqp} and differential equations \Ref{dif eqn p}. To construct
a solution of these equations  choose a master polynomial
$\Phi(z,t) = \prod_{j\in J} f_j(z,t)^{\frac{p-1}2}$,
consider the Taylor expansion of the polynomial
\bean
\label{pol F}
F(z,t) = {\sum}_{{\rm independent } \atop \{j_1 < \dots < j_k\} \subset J }
 \Phi(z,t) \frac{d_{j_1,\dots,j_k}}
{f_{j_1}(z,t)\dots f_{j_k}(z,t)}\, F(H_{j_1},\dots,H_{j_k})
\eean
at $t=0$, and obtain the solution
\bean
\label{[i]}
[I]^{p-1,\dots,p-1}(z)=
{\sum}_{{\rm independent } \atop \{j_1 < \dots < j_k\} \subset J }
[I]^{p-1,\dots,p-1}_{j_1,\dots,j_k}
(z)\, F(H_{j_1},\dots,H_{j_k})
\eean
 of the algebraic equations \Ref{sieqp} and  differential equations \Ref{dif eqn p}
 by taking the coefficient
of $(t_1\dots t_k)^{p-1}$ of the Taylor expansion, see Theorem \ref{thm m}.

Let $x=(x_1,\dots,x_n)\in\FFF_p^n$. Let 
$\Ga(x)$ be
 the affine hypersurface
\bean
\label{ec}
y^2 = {\prod}_{j\in J} [f]_j(x,t)
\eean
over $\FFF_p$.  Recall that $[f]_j(x,t) = [b^1_j]t_1+\dots +[b^k_j]t_k + x_j$,
 where $[\,]: \Z\to \FFF_p$ is the natural projection.

 For a rational function $h : \Ga(x)\to \FFF_p$  define the integral  
\bean
\label{int}
\int_{\Ga(x)} h = {\sum}_{P\in\Ga(x)}' h(P),
\eean
as the sum over all points $P\in\Ga(x)$ with well-defined $h(P)$.

\begin{thm}
\label{thm pts p}

Let 
\bean
\label{INeq}
n\frac{p-1}2-k<(k+1)(p-1).
\eean 
Let $[I]^{p-1,\dots,p-1}(x)$ 
be  the vector  in
\Ref{[i]} at $z=x$. Then for any independent $\{j_1<\dots<j_k\}\subset J$ we have
\bean
\label{pts=sol1} 
\int_{\Ga(x)} \frac{[d_{j_1,\dots,j_k}]}
{{\prod}_{l=1}^k[f]_{j_l}(x,t)}
& = &  (-1)^k [I]^{p-1,\dots,p-1}_{j_1,\dots, j_k}(x).
\eean
\end{thm}

This theorem is a generalization of Theorem \ref{thm pts}.
Theorem \ref{thm pts p} says that the integrals in the left-hand side of \Ref{pts=sol} are polynomials
in $x_1,\dots,x_n\in\FFF_p$ and satisfy 
the algebraic equations \Ref{sieqp} and  differential equations \Ref{dif eqn p}. 

\begin{proof}
It is easy to see that 
\bean
\label{sum}
&&
\int_{\Ga(x)}\frac{[d_{j_1,\dots,j_k}]}
{\prod_{l=1}^k[f]_{j_l}(x,t)}
\\
&&
\phantom{aaaa}
\notag
 = 
{\sum}_{t\in\FFF_p^k, \atop
\prod_{l=1}^k[f]_{j_l}(x,t)\ne 0} 
\frac{[d_{j_1,\dots, j_k}]}{\prod_{l=1}^k[f]_{j_l}(x,t)}
+{\sum}_{t\in\FFF_p^k} \frac{[d_{j_1,\dots,j_k}]}{\prod_{l=1}^k[f]_{j_l}(x,t)}
\Big({\prod}_{j\in J} [f]_j(x,t)\Big)^{\frac{p-1}2},
\eean
cf. \cite{Ma} and \cite{Cl}.

\begin{lem}
\label{lem =0} The first sum in the right-hand side of \Ref{sum} equals zero.

\end{lem}

\begin{proof}
Since $p$ is good and $\{j_1,\dots,j_k\}$ is independent, we have
$[d_{j_1,\dots, j_k}]\ne 0$. Hence we may choose $s_l:= [f]_{j_l}(x,t)$,
$l=1,\dots,k$, to be affine coordinates on $\FFF_p^k$. Then
\bea
{\sum}_{t\in\FFF_p^k, \atop
\prod_{l=1}^k[f]_{j_l}(x,t)\ne 0} 
\frac{[d_{j_1,\dots, j_k}]}{\prod_{l=1}^k[f]_{j_l}(x,t)}
= {\sum}_{s\in\FFF_p^k, \atop
s_1\dots s_k\ne 0} 
\frac{[d_{j_1,\dots, j_k}]}{s_1\dots s_k}
={\sum}_{s\in\FFF_p^k}
[d_{j_1,\dots, j_k}] (s_1\dots s_k)^{p-2} = 0.
\eea
\end{proof}

\begin{lem}
\label{lem =01} The second sum in the right-hand side of \Ref{sum} equals 
$(-1)^k[I]^{p-1,\dots,p-1}_{j_1,\dots, j_k}(x)$.

\end{lem}

\begin{proof}
The inequality \Ref{ineq} takes the form \Ref{INeq}.
 Now the lemma follows from Theorem \ref{thm 2int}.
\end{proof}
Theorem \Ref{thm pts p} is proved.
\end{proof}

\begin{rem}
Inequality \Ref{INeq} holds if $n\leq 2k+2$ independently of $p$.

\end{rem}

\subsection{Counting on $\ka$-folded covers}
\label{cpw3}
Let $\ka>2$ be a positive integer. 
As in Sections \ref{AnaZ}  and \ref{cpww} consider $n$ nonzero linear functions on $\C^k$,
$g_j = b_j^1t_1+ \dots+b_j^kt_k$, $j\in J$,
 with integer coefficients $b_j^i\in \Z$.
Let a prime integer $p$ be good  with respect to $(g_j)_{j\in J}$ and $\ka\big|(p-1)$.

Assume that all weights $(a_j)_{j\in J}$ are equal to $-(\ka-1)$. 
Under these assumptions consider
the algebraic equations \Ref{sieqp} and  differential equations \Ref{dif eqn p}. To construct
a solution of these equations choose a master polynomial
$\Phi(z,t) = \prod_{j\in J} f_j(z,t)^{(\ka-1)\frac{p-1}\ka}$,
consider the Taylor expansion at $t=0$ of the polynomial $F(z,t)$ in \Ref{pol F} and obtain the solution
$[I]^{p-1,\dots,p-1}(z)$
 of the algebraic equations \Ref{sieqp} and differential equations \Ref{dif eqn p}
 by taking the coefficient
of $(t_1\dots t_k)^{p-1}$ of the Taylor expansion, see Theorem \ref{thm m} and formula \Ref{[i]}.

Let $x=(x_1,\dots,x_n)\in\FFF_p^n$. Let 
$\Ga(x)$ be
 the affine hypersurface
\bean
\label{ec}
y^\ka = {\prod}_{j\in J} [f]_j(x,t)
\eean
over $\FFF_p$, cf. Section \ref{cpww}.

\begin{thm}
\label{thm pts p3} Let a prime integer $p$ be good with respect to $(g_j)_{j\in J}$. Let $\ka\big|(p-1)$ 
and $\ka\ne p-1$.
Let
\bean
\label{Ineq}
n(\ka-1)\frac{p-1}\ka -k < (k+1)(p-1),
\qquad
n(\ka-2)\frac{p-2}\ka -k < k(p-1).
\eean
 Then for any independent $\{ j_1<\dots<j_k\}\subset J$ we have
\bean
\label{pts=sol2} 
\int_{\Ga(x)} \frac{[d_{j_1,\dots,j_k}]}
{\prod_{l=1}^k[f]_{j_l}(x,t)}
& = &  (-1)^k [I]^{p-1,\dots,p-1}_{j_1,\dots, j_k}(x).
\eean
\end{thm}

This theorem is a generalization of \cite[Example 1.7]{SV2} and Theorem \ref{thm pts p}.

\begin{proof}
It is easy to see that 
\bean
\label{sum3}
\int_{\Ga(x)}\frac{[d_{j_1,\dots,j_k}]}
{\prod_{l=1}^k[f]_{j_l}(x,t)}
& =& 
{\sum}_{t\in\FFF_p^k, \atop
\prod_{l=1}^k[f]_{j_l}(x,t)\ne 0} 
\frac{[d_{j_1,\dots, j_k}]}{\prod_{l=1}^k[f]_{j_l}(x,t)}
\\
\notag
&+&
{\sum}_{\ell=1}^{\ka-1}{\sum}_{t\in\FFF_p^k} \frac{[d_{j_1,\dots,j_k}]}{\prod_{l=1}^k[f]_{j_l}(x,t)}
\Big({\prod}_{j\in J} [f]_j(x,t)\Big)^{\ell\frac{p-1}\ka}.
\eean
The first sum in the right-hand side equals zero by Lemma \ref{lem =0}.

Consider the Taylor expansion of the polynomial 
$\frac{[d_{j_1,\dots,j_k}]}{\prod_{l=1}^k[f]_{j_l}(x,t)}
\Big({\prod}_{j\in J} [f]_j(x,t)\Big)^{\ell\frac{p-1}\ka}$. Consider the monomials
of the form $t_1^{l_1(p-1)}\dots t_k^{l_k(p-1)}$, where $l_1,\dots,l_k$ are positive integers.
If $\ell \leq \ka-2$, the second inequality in \Ref{Ineq} implies
that the coefficients of such
monomials in the Taylor expansion are all equal to zero and hence
the sum 
$ \sum_{t\in\FFF_p^k} \frac{[d_{j_1,\dots,j_k}]}{\prod_{l=1}^k[f]_{j_l}(x,t)}
\Big({\prod}_{j\in J} [f]_j(x,t)\Big)^{\ell\frac{p-1}\ka}$ equals zero for $\ell\leq \ka-2$.
If $\ell=\ka-1$, the first inequality in \Ref{Ineq} implies that 
among the monomials of the form $t_1^{l_1(p-1)}\dots t_k^{l_k(p-1)}$
only
 $t_1^{p-1}\dots t_k^{p-1}$ may appear with a
nonzero coefficient in the Tylor expansion. Hence 
\bea
{\sum}_{t\in\FFF_p^k} \frac{[d_{j_1,\dots,j_k}]}{\prod_{l=1}^k[f]_{j_l}(x,t)}
\Big({\prod}_{j\in J} [f]_j(x,t)\Big)^{(\ka-1)\frac{p-1}\ka}= (-1)^k [I]^{p-1,\dots,p-1}_{j_1,\dots,j_k}(x)
\eea
by Theorem \ref{thm 2int}.
The theorem is proved.
\end{proof}

\begin{rem} Inequalities \Ref{Ineq} are implied by the system of  inequalities
\bean
\label{Ineq1}
\frac\ka{\ka-1}(k+1)\geq n,
\qquad
\frac\ka{\ka-2}k\geq n
\eean
independent of $p$.
If $k\geq\ka-2$ , then the inequality $\frac\ka{\ka-1}(k+1)\geq n$ implies 
inequality $\frac\ka{\ka-2}k\geq n$
and is enough for Theorem \ref{thm pts p3} to hold.
In particular, if $k\geq \ka-2$, then $n=k+2$ is admissible.
\end{rem}

\section{Bethe ansatz}
\label{sec BA}

The goal of the Bethe ansatz is to construct mutual eigenvectors of the geometric Hamiltonians
$([K]_i(x))_{i\in J}$ defined in \Ref{Kjp}.

\smallskip
As in Sections \ref{AnaZ}, \ref{cpww} and \ref{cpw3} consider $n$ nonzero linear functions on $\C^k$,
$g_j = b_j^1t_1+ \dots+b_j^kt_k$, $j\in J$,
 with integer coefficients $b_j^i\in \Z$.
Let a prime integer $p$ be good  with respect to $(g_j)_{j\in J}$.
Let $a=(a_j)_{j\in J}$ be nonzero integer weights
$a_j\in\Z$, $a_j\ne 0$. 

Recall the functions $[f]_C(z) =\sum_{i\in J} [\la_C^i]z_i$ with $C\in \frak C_{\FFF_p}$.
Assume that $x=(x_1,\dots,x_n)\in \FFF_p^n$ is such that 
$f_C(x)\ne 0$ for any $C\in \frak C_{\FFF_p}$. Then
$([K_i](x))_{i\in J}$ are well-defined linear operators on $V_{\FFF_p}$.

Introduce the system of the {\it Bethe ansatz equations}
\bean
\label{BAE}
{\sum}_{j\in J} b^i_j\frac{[a_j]}{[f]_j(x,t)}= 0,\qquad i=1,\dots,k,
\eean
with respect to the unknown $t=(t_1,\dots, t_k)\in \FFF_p^k$.

\begin{thm} 
\label{thm BAp}
If $t^0\in \FFF_p^{k}$ is a solution of equations
\Ref{BAE}, then
the vector
\bean
\label{pol FB}
F(x,t^0) = {\sum}_{{\rm independent } \atop \{j_1 < \dots < j_k\} \subset J }
\frac{[d_{j_1,\dots,j_k}]}
{[f]_{j_1}(x,t^0)\dots [f]_{j_k}(x,t^0)}\, F(H_{j_1},\dots,H_{j_k})
\eean
satisfies equations \Ref{sieqp} and  is an eigenvector of the geometric Hamiltonians:
\bean
\label{eig}
[K]_i(x) F(x,t^0) = \frac{[a_i]}{[f]_i(x,t^0)} F(x,t^0),
\qquad i=1,\dots, n.
\eean
\end{thm}

\begin{proof} 
Equations \Ref{sieqp} for $F(x,t^0)$ follow from equations 
\Ref{id1} and \Ref{id111} reduced modulo $p$. 
Since $t^0$ is a solution of \Ref{BAE} the left-hand side of \Ref{id111} equals zero.
Equations \Ref{eig} are straightforward corollaries of the key identity \Ref{fHp}, see
the proof of \cite[Theorems  2.1]{V5}.
\end{proof}

\begin{cor}
\label{cor orth}

 Let $t^0$ and $t^1$ be  distinct solutions of the Bethe ansatz equations
\Ref{BAE}, then  $[S]^{(a)}(F(x,t^0),F(x,t^1))=0$.
\end{cor}

\begin{proof} Since $t^0\ne t^1$, there exists $i$ such that $[f]_i(x,t^0)\ne [f]_i(x,t^1)$. Hence
$[K]_i(x)$ has distinct eigenvalues on $F(x,t^0)$, $F(x,t^1)$, but $[K_i](x)$ is symmetric,
\bea
[S]^{(a)}([K]_i(x)F(x,t^0), F(x,t^1))=[S]^{(a)}(F(x,t^0), [K]_i(x)F(x,t^1)).
\eea
\end{proof}

\end{document}